\documentclass{article}
\usepackage{amscd, amssymb, graphics}
\usepackage{amsmath,amsfonts,amsthm,enumerate}
\usepackage{latexsym}
\usepackage{graphicx,psfrag}
\usepackage{mathrsfs}
\usepackage[all]{xy}

\usepackage{hyperref}

\newtheorem{theorem}{Theorem}[section]
\newtheorem{lemma}[theorem]{Lemma}
\newtheorem{prop}[theorem]{Proposition}
\newtheorem{cor}[theorem]{Corollary}

\theoremstyle{definition}
\newtheorem{definition}[theorem]{Definition}
\newtheorem{example}[theorem]{Example}

\newtheorem{conj}[theorem]{Conjecture}

\theoremstyle{remark}
\newtheorem{remark}[theorem]{Remark}

\numberwithin{equation}{section}

\newcommand{\set}[2]{{\{ {#1}\ |\ {#2}\}}}

\newcommand{\RR}{{\mathbb R}}

\newcommand{\CC}{{\mathbb C}}
\newcommand{\ZZ}{{\mathbb Z}}

\newcommand{\Rfour}{{{\mathbb R}^4}}

\newcommand{\GL}{{\rm GL}}
\newcommand{\SL}{{\rm SL}}
\newcommand{\SO}{{\rm SO}}
\newcommand{\Spin}{{\rm Spin}}

\newcommand{\Arf}{{\mathrm{Arf}}}
\newcommand{\id}{{\rm id}}
\newcommand{\Aut}{{\rm Aut}}
\newcommand{\Mod}{{\rm Mod}}

\newcommand{\diff}{{\mathtt{Diff}}}
\newcommand{\topo}{{\mathtt{Top}}}
\newcommand{\pl}{{\mathtt{PL}}}
\newcommand{\cat}{{\mathcal{C}}}
\newcommand{\mcg}{{\mathrm{MCG}}}
\newcommand{\esg}{{\mathscr{E}}}

\begin{document}

\title{Spin structures and codimension-two homeomorphism extensions}
\author{Fan Ding, Yi Liu,  Shicheng Wang, Jiangang Yao}
\maketitle

\begin{abstract}
Let $\imath: M\to \RR^{p+2}$ be a smooth embedding
from a connected, oriented, closed $p$-dimesional 
smooth manifold to $\RR^{p+2}$, 
then there is a spin structure $\imath^\sharp(\varsigma^{p+2})$ on $M$
canonically induced from the embedding. If an
orientation-preserving diffeomorphism $\tau$ of $M$ extends over
$\imath$ as an orientation-preserving topological homeomorphism of
$\RR^{p+2}$, then $\tau$ preserves the induced spin structure.

Let $\esg_\cat(\imath)$ be the subgroup of the $\cat$-mapping class group
$\mcg_\cat(M)$ consisting of elements whose representatives extend over $\RR^{p+2}$
as orientation-preserving $\cat$-homeomorphisms, where $\cat=\topo$, $\pl$ or $\diff$.
The invariance of $\imath^\sharp(\varsigma^{p+2})$ gives nontrivial 
lower bounds to $[\mcg_\cat(M):\esg_\cat(\imath)]$ in various special cases.
We apply this to embedded surfaces in $\RR^4$ and embedded $p$-dimensional tori 
in $\RR^{p+2}$. In particular, in these cases the index lower bounds for $\esg_\topo(\imath)$
are achieved for unknotted embeddings.
\end{abstract}


\section{Introduction}

Let $M$ be a connected, oriented, closed $p$-dimensional smooth manifold, and
$\imath:M\hookrightarrow \RR^{p+2}$ be a smooth embedding. We are concerned with 
the question: `how many mapping classes of $M$ extend over $\RR^{p+2}$?'
Regarding to different possible flavors of this question, we shall
write $\topo$ (resp. $\pl$, or $\diff$) for the category
of topological (resp. PL, or smooth) manifolds
with continuous (resp. PL, or smooth) maps 
as morphisms, and generally write $\cat$ for any of these categories. We
speak of $\cat$-manifolds, $\cat$-homeomorphisms, $\cat$-isotopies, 
etc. in the usual sense.

With notations above, denote $\mcg_\cat(M)=\pi_0\,{\rm Homeo}^+_\cat(M)$ for the $\cat$-mapping-class-group of $M$, i.e. the group of $\cat$-isotopy classes of orientation-preserving $\cat$-self-homeomorphisms on $M$. A class $[\tau]\in\mcg_\cat(M)$ is called \emph{$\cat$-extendable} over $\imath$ if for some (hence any, cf. Lemma \ref{extendTogether}) representative $\tau$, there is an orientation-preserving $\cat$-self-homeomorphism $\tilde\tau$ of $\RR^{p+2}$ such that $\imath\circ\tau=\tilde{\tau}\circ\imath$.
We define the \emph{$\cat$-extendable subgroup} with respect to $\imath$ as:
$$\esg_\cat(\imath)=\set{[\tau]\in\mcg_\cat(M)}{\text{$\tau$ is $\cat$-extendable over $\imath$.}}.$$

Now the question makes sense by asking what is the index of $\esg_\cat(\imath)\leq\mcg_\cat(M)$. In this paper, we prove the following criterion.

Let $\varsigma^{p+2}$ by the canonical spin structure on $\RR^{p+2}$.
For any smooth embedding $\imath: M\to \RR^{p+2}$, there is a canonically induced spin structure 
$\imath^\sharp(\varsigma^{p+2})$ on $M$ (Definition \ref{inducedSpin}).

\begin{prop}\label{spin-obstruction-top}
For any smooth embedding $\imath: M\to \RR^{p+2}$, the induced spin structure 
$\imath^\sharp(\varsigma^{p+2})$ is on
$M$ is null spin-cobordant, and is invariant under any orientation-preserving self-diffeomorphism of $M$ which extends
over $\imath$  as an orientation-preserving topological self-homeomorphism of $\RR^{p+2}$.
\end{prop}

In fact, $\imath^\sharp(\varsigma^{p+2})$ is naturally induced as the boundary of a spin structure on a smooth Seifert hypersurface $\Sigma$ of $\imath(M)$. Proposition \ref{spin-obstruction-top} allows us to find nontrivial lower bounds of $[\mcg_\cat(M):\esg_\cat(\imath)]$ in certain cases. We may even compute the index for some specific embeddings. In this paper, we apply the criterion to smoothly embedded surfaces in $\RR^4$ and certain smoothly embedded $p$-dimensional torus in $\RR^{p+2}$.

\begin{theorem}\label{main-extend-F_g}
For any smooth embedding $\imath: F_g\hookrightarrow \Rfour$ of the closed oriented
surface of genus $g$ into $\Rfour$,
$$[{\mcg}_\topo(F_g):\esg_\topo(\imath)]\geq 2^{2g-1}+2^{g-1}.$$
\end{theorem}

\begin{remark} In principle, one should be able to 
derive the smooth version that $[\mcg_\diff(F_g):\esg_\diff(\imath)]\geq
2^{2g-1}+2^{g-1}$ from the invariance of the Rokhlin
 quadratic form (\cite{Ro}). However, this, also the PL version, 
is immediately implied by Theorem \ref{main-extend-F_g} 
as the $\mcg_\cat(F_g)$'s are canonically isomorphic for $\cat$ being $\diff$, $\pl$, and $\topo$. The lower bound in Theorem \ref{main-extend-F_g}
 is sharp for any unknotted embedding of $F_g$ in $\Rfour$, namely
which bounds a smoothly embedded handlebody of genus $g$ in $\Rfour$, following from an
intensive construction of Susumu Hirose (\cite{Hi}, cf. also (\cite{Mo} for $g=1$).
See Section \ref{Sec-F_g} for details.
\end{remark}

Another interesting fact follows from the proof of Theorem \ref{main-extend-F_g}.

\begin{cor}\label{noExtension} For any $g\geq 1$, there exists $[\tau]\in\mcg_\topo(F_g)$ which is not homeomorphically extendable over any smooth embedding $\imath:F_g\hookrightarrow\RR^4$. \end{cor}

Denote the standard $p$-dimensional smooth torus $S^1\times\cdots\times S^1\,(p\textrm{ copies})$ as $T^p$. The structure of $\mcg_\cat(T^p)$ is fairly well-understood except for $p=4$, thank to the work of Allen Hatcher et al, (see Section \ref{Sec-T^p} for details), which makes the following theorem attainable.

\begin{theorem}\label{main-extend-T^p} For $p\geq 1$, suppose 
$\imath:T^p\hookrightarrow\RR^{p+2}$ is a smooth embedding whose induced spin structure $\imath^\sharp(\varsigma^{p+2})$
on $T^p$ is not the Lie-group spin structure, then:
$$[\mcg_\topo(T^p):\esg_\topo(\imath)]\geq 2^p-1.$$
Moreover, the lower bound is realized by \emph{unknotted} embeddings (Definition \ref{ukT}).
\end{theorem}

\begin{remark} A parallel proof shows $[\mcg_\cat(T^p):\esg_\cat(\imath)]\geq 2^p-1$, cf. Lemma \ref{actionT^p}. For $p\neq4$, this also follows from the fact that the natual forgetting functor $\diff\to\pl\to\topo$ yields epimorphisms on
$\mcg_\cat(T^p)$ and homomorphisms on $\esg_\cat(T^p)$, and hence $[\mcg_\diff(T^p):\esg_\diff(\imath)]\geq[\mcg_\pl(T^p):\esg_\pl(\imath)]\geq[\mcg_\topo(T^p):\esg_\topo(\imath)]$.
\end{remark}

\begin{cor}\label{extend-T^3} If a smoothly embedded $3$-torus $T^3$ in $\RR^5$ has a (hence any) 
smooth Seifert hypersurface $\Sigma^4$ of signature $0$ modulo $16$, then $[\mcg_\topo(T^3):\esg_\topo(\imath)]\geq 7$.\end{cor}

Note any smooth Seifert hypersurface in this case must have signature $0$ modulo $8$, (cf. \cite[Proposition 6.1]{SST}).

\begin{cor}\label{ukT-DiffPL} If $\imath:T^p\hookrightarrow\RR^{p+2}$ is an unknotted embedding, then $[\mcg_\diff(T^p):\esg_\diff(\imath)]$ and $[\mcg_\pl(T^p):\esg_\pl(\imath)]$ are finite but at least $2^p-1$.\end{cor}

\begin{remark} While these indices are indeed $2^p-1$ when $p\leq 3$, it is still an interesting question figuring out the indices when $p>3$.   \end{remark}

It is pretty easy to find examples 
where the assumption of Theorem \ref{main-extend-T^p} holds, but with some effort one can still
find smooth embeddings $\imath:T^p\hookrightarrow\RR^{p+2}$ which induce the Lie-group spin structure on $T^p$ when $p\geq 3$, (cf. \cite{SST} for such an example). Unfortunately, Theorem \ref{main-extend-T^p} says nothing about that case. We make the following conjecture:

\begin{conj}
For any smooth embedding $\imath:T^p\hookrightarrow\RR^{p+2}$, $\esg_\topo(\imath)$ is a proper
subgroup of $\mcg_{\topo}(T^p)$.
\end{conj}

In Section
\ref{Sec-spinPrelim}, we recall some preliminary material about spin
structures in terms of trivilizations. In Section \ref{Sec-inducedSpin},
we introduce $\imath^\sharp(\varsigma^{p+2})$ (Definition \ref{inducedSpin} using
Seifert hypersurfaces, and prove Proposition \ref{spin-obstruction-top}. 
In Section \ref{Sec-F_g}, we consider embedded surfaces in $\RR^4$
and using the action of $\mcg_\cat(F_g)$ on the space of spin structures $\mathcal{S}(F_g)$
on $F_g$ to prove Theorem \ref{main-extend-F_g} and Corollary \ref{noExtension}.
In Section \ref{Sec-T^p}, we consider embedded $T^p$ in $\RR^{p+2}$.
We first review some results of Hatcher about the structure of $\mcg_\cat(T^p)$ 
for $p\neq4$,
then prove Theorem \ref{main-extend-T^p} and its corollaries by studying
the action of $\mcg_\cat(T^p)$ on $\mathcal{S}(T^p)$.

\bigskip\noindent\textbf{Acknowledgement}. The first and the third authors are partially supported by grant
No.10631060 of the National Natural Science Foundation of China. We
also thank Robert Edwards, Andrey Gogolev, Jonathan Hillman, Robion Kirby, and
Hongbin Sun for helpful communications.

\section{Spin structure preliminaries}\label{Sec-spinPrelim}

In this section, we recall some basic facts about spin structures,
cf. \cite[Chapter IV]{Ki}, \cite{Mi63}.

Spin structures of a rank $n$ vector bundle $\xi$ over a CW complex
$X$ can phrased with trivialization, i.e. framing. $\xi$ can be
endowed with a spin structure if $E\oplus\epsilon^k$ has a
trivialization over the $1$-skeleton which may extend over the
$2$-skeleton, (if $n\geq 3$, $k=0$; if $n=2$, $k=1$; if $n=1$,
$k=2$), and a spin structure is a homotopy class of such
trivializations over the $1$-skeleton. For a CW complex $X$, we use
$X^{(i)}$ to denote its $i$-skeleton.

\begin{lemma}\label{eqvSpin} Suppose $\xi$ is a rank $n$ vector bundle over a topological space
$M$ that admits CW structures. Then there is a natural bijection
between the sets of spin structures corresponding to different CW
structures.\end{lemma}

\begin{proof}
Suppose $X_0$, $X_1$ are two CW structures on $M$, and let
$\sigma_i$ be a spin structure of $\xi$ with respect to $X_i$
($i=0,1$). There is a natural \emph{difference} homomorphism:
$$\sigma_1-\sigma_0:\pi_1(M)\to\ZZ_2,$$
defined as follows. For any $[\alpha]\in \pi_1(M)$, let
$\gamma_0,\gamma_1:S^1\to M$ be two closed paths in
$X^{(1)}_0,X^{(1)}_1$ respectively, both basepoint-freely homotopic
to $\alpha$ in $M$. Let $\gamma_t:S^1\to M$ ($t\in I$) be any
homotopy between $\gamma_0$ and $\gamma_1$, and pick a
trivialization $\varsigma$ of $\epsilon^{n+k}|_{S^1\times I}$. If
$\gamma:S^1\times I\to M$ extends to a bundle map
$\tilde{\gamma}:\epsilon^{n+k}|_{S^1\times I}\to
\xi\oplus\epsilon^k|_M$, such that $\tilde{\gamma}_*(\varsigma|_{S^1\times
0})\simeq \sigma_0$, $\tilde{\gamma}_*(\varsigma|_{S^1\times 1})\simeq
\sigma_1$, then we define $(\sigma_1-\sigma_0)([\alpha])=0$,
otherwise $(\sigma_1-\sigma_0)([\alpha])\ne 0$. It is easy to 
see that $\sigma_1-\sigma_0$ is a well-defined homomorphism, and
$(\sigma_2-\sigma_1)-(\sigma_1-\sigma_0)=\sigma_2-\sigma_0$.

Thus, we define two spin structures with respect to possibly
different spin structures to be \emph{equivalent} if their
difference is zero. If $X_0$ and $X_1$ happen to be the same, it is
clear that $\sigma_0$ and $\sigma_1$ are equivalent if and only if
they are equal by definition. Moreover, for any CW complex
structure, the space of spin structures forms an affine
$H^1(M;\ZZ_2)\cong {\rm Hom}(\pi_1(M),\ZZ_2)$, precisely as
described by the difference homomorphism. Thus, if
$\sigma_1-\sigma_0$ is not zero, we find exactly one $\sigma'_1$
with respect to $X_1$, such that $\sigma'_1-\sigma_0$ is zero. This
implies that the spaces of spin structures corresponding to
different CW structures are in natural bijection to each other,
namely according to the equivalence.\end{proof}

For a rank $n$ vector bundle $\xi$ over a CW space $M$, a spin
structure on $\xi$ is known as with respect to any CW structure on
$M$, up to the natural equivalence. A spin structure of a smooth
manifold $M$ is a spin structure of its tangent bundle. For any
closed path $\alpha$ on $M$, we may also restrict a spin structure
$\sigma$ of $M$ to $\sigma|_\alpha$, namely picking a trivialization
of $\xi|_\alpha$ which extends to a trivialization on (some)
$1$-skeleton equivalent to $\sigma$. It is clear that
two spin structures $\sigma_1,\sigma_0$ of $\xi$ are equivalent
if and only if the trivialization $\sigma_1|_\alpha\simeq \sigma_0|_\alpha$ for
any closed path $\alpha$ on $M$. 
For an oriented smooth manifold $M$, $M$ has a spin structure if and only if $w_2(M)=0$, and 
when $M$ has
spin structures, the space $\mathcal{S}(M)$ of spin structures on $M$ is an affine $H^1(M;\ZZ_2)$.

If $\xi=\xi'\oplus \xi''$ are bundles over a CW space $X$, then
spin structures on any two determine a spin structure of
the third, so that $\sigma\simeq\sigma'\oplus\sigma''|_{X^{(1)}}$.
The less trivial direction that $\sigma,\sigma'$ determine $\sigma''$
follows from the general fact: if $\xi=\eta\oplus\xi'$ is a trivial
 $(n+k)$-vector bundle where $\eta$, $\xi'$ 
has rank $n$, $k$, respectively, and if $\xi$, $\xi'$ are both trivialized
over $X^{(n-1)}$, then there is a complementary trivialization of $\eta$ over $X^{(n-1)}$,
which is unique over $X^{(n-2)}$ up to homotopy, (cf. \cite[p. 33]{Ki}).

For a spin manifold $M$ 
with boundary $\partial M$, $\partial M$ has a natural spin 
structure induced from the spin structure of $M$ 
and the (inward) normal vector of $\partial M$ in $M$. A manifold is called a \emph{spin
boundary} if there is a spin manifold bounding it, inducing its spin structure.
For example, the circle $S^1$ has two spin structures: one spin-bounds the spin $D^2$, 
and the other is the Lie-group spin structure which is not a spin-boundary.

\section{Invariant induced spin structure}\label{Sec-inducedSpin}

In this section, we introduce the induced spin structures for closed oriented codimension-$2$
smooth submanifolds of $\RR^{p+2}$, and prove Proposition \ref{spin-obstruction-top}.

Suppose $\imath:M\hookrightarrow\RR^{p+2}$ is a connected, closed, oriented $p$-dimensional 
smooth submanifold of $\RR^{p+2}$, $p\geq 1$. Since any closed oriented
smooth submanifold of $\RR^{p+2}$ has trivial Euler class
(\cite[Corollary 11.4]{MS}), the normal bundle of $M$ in
$\RR^{p+2}$ is trivial as $M$ is codimension $2$. On the other hand,
it is well-known that there exists
a \emph{Seifert hypersurface} $\Sigma\subset\RR^{p+2}$ of $\imath(M)$, 
namely, a compact connected oriented $(p+1)$-dimensional smooth submanifold
such that $\partial\Sigma=\imath(M)$, (cf. for example, \cite[Lemma 2.2]{Er}).

Let $W$ be an inward normal vector field of $\imath(M)$ in $\Sigma$ (say, 
w.r.t some compatible Riemannian metric on a collar), and $H$ be a normal vector field of $\Sigma$ in $\RR^{p+2}$ over $\imath(M)$, such that the orientation $(W,H)$ of the normal bundle $N_{\RR^{p+2}}(\imath(M))$ and the orientation
of $M$ match up to that the canonical orientation of $\RR^{p+2}$. The trivialization
$(W,H)$ of $N_{\RR^{p+2}}(\imath(M))$ defines a spin structure $\sigma$ of $N_{\RR^{p+2}}(\imath(M))$,
and the canonical spin structure $\varsigma^{p+2}$ of $\RR^{p+2}$ restricts to a spin structure
on $T\RR^{p+2}|_{\imath(M)}$. As $T\RR^{p+2}|_{\imath(M)}=\imath_*(TM)\oplus N_{\RR^{p+2}}(\imath(M))$,
we conclude there is a complementary spin structure $\sigma^\perp$ of $\imath_*(TM)$ such that: $$\sigma^\perp\oplus\sigma=\varsigma^{p+2},$$ (i.e. as trivializations $\sigma^\perp\oplus\sigma\simeq\varsigma$ over ${M^{(1)}}$).

\begin{lemma}\label{independent} The spin structure $\sigma^\perp$ on $\imath(M)$ is independent of the choice of $\Sigma$ and $(W,H)$. \end{lemma}

\begin{proof} It suffices to show $\sigma$ is independent of the choice of $\Sigma$ and $(W,H)$. In fact, we show for any two choices $\Sigma$, $(W,H)$ and $\Sigma'$, $(W',H')$, the trivializations $(W,H)\simeq (W',H')$ over $\imath(M)$.

First observe that any loop $\alpha$ on $\imath(M)$, when pushed into $\mathring\Sigma$ along $W$, becomes null-homologous in $\RR^{p+2}\setminus \imath(M)$.
To see this, consider the map $f:\RR^{p+2}\setminus \imath(M) \to S^1$ defined as follows: take a tubular neighborhood $\mathcal{N}(\mathring\Sigma)$, where $\mathring\Sigma$ is the interior of $\Sigma$, and let $f|:\mathcal{N}(\mathring\Sigma)\to S^1$ to be the composition:
$\mathcal{N}(\mathring\Sigma)\cong \mathring\Sigma\times I\stackrel{p}\longrightarrow I\stackrel{q}\longrightarrow I/\partial I\cong S^1$,
where $p$ is the second-factor projection and $q$ is the quotient map; then extend $f|$ to $f:\RR^{p+2}\setminus \imath(M) \to S^1$ by the constant map. Then $f_*: H_1(\RR^{p+2}\setminus \imath(M))\to H_1(S^1)$ is isomorphic, but the push-off of $\alpha$ along $W$ is mapped to $0\in H_1(S^1)$, so it is null-homologous in $\RR^{p+2}\setminus\imath(M)$.

Now $(W,H)$ and $(W',H')$ differ pointwisely by an element of $\GL^+(2,\RR)$, namely for any $x\in M$, $(W',H')|_x=(W,H)|_x\cdot R(x)$ for some $R(x)\in \GL^+(2,\RR)$. This gives a map $R:M\to \GL^+(2,\RR)$. If for some loop $\alpha:S^1\to M$, $R\circ\alpha$ were not null-homotopic in $\GL^+(2,\RR)$, then the push-offs of $\alpha$ along $W$ and $W'$ would differ by a non-zero multiple of the meridian $\mu$, namely the loop which bounds a normal disk of $\imath(M)$. Since $\mu$ is the generator of $H_1(\RR^{p+2}\setminus\imath(M))\cong\ZZ$ by the Alexander duality, the two push-offs would not be both null-homologous in $\RR^{p+2}\setminus\imath(M)$, which is a contradiction. Thus $R_\sharp:\pi_1(M)\to\pi_1(\GL^+(2,\RR))$ is trivial. We conclude that $R$ is homotopic to the constant identity map as $\pi_i(\GL^+(2,\RR))\cong\pi_i(S^1)=0$ for $i\geq 2$.
This implies $(W,H)\simeq(W',H')$ over $M$.
\end{proof}

Lemma \ref{independent} allows us to make the following definition.

\begin{definition}\label{inducedSpin} For a smooth embedding $\imath:M\hookrightarrow\RR^{p+2}$ of a connected, closed, oriented $p$-dimensional 
smooth manifold $M$ into $\RR^{p+2}$, we define the \emph{induced spin structure} as:
$$\imath^\sharp(\varsigma^{p+2})=\imath^*(\sigma^\perp),$$
where $\sigma^\perp$ is as described above.
\end{definition}

\begin{proof}[{Proof of Proposition \ref{spin-obstruction-top}}]
We first show $\imath^\sharp(\varsigma^{p+2})$ is null spin-cobordant, or equivalently that $\sigma^\perp$ is a spin boundary. In fact, for a Seifert hypersurface $\Sigma$ of $\imath(M)$, the normal verctor field $H$ of $\Sigma$ in $\RR^{p+2}$ defines a spin structure $\sigma_H$ on the normal bundle $N_{\RR^{p+2}}(\Sigma)$, so there is a spin structure $\sigma_H^\perp$ on $T\Sigma$ such that $\sigma_H^\perp\oplus\sigma_H=\varsigma^{p+2}$ on $T\Sigma\oplus N_{\RR^{p+2}}(\Sigma)= T\RR^{p+2}|_\Sigma$. The spin boundary of $(\Sigma, \sigma_H^\perp)$ is clearly $(M,\sigma^\perp)$ by the construction.

We next prove the invariance of $\imath^\sharp(\varsigma^{p+2})$ under homeomorphically extendable self-diffeomorphisms.
Specifically, for an orientation-preserving self-diffeomorphism $\tau:M\to M$ which extends over $\imath$ as an orientation-preserving topological self-homeomorphism $\tilde\tau$ of $\RR^{p+2}$, we must show $\imath^\sharp(\varsigma^{p+2})$ equals $\tau^*(\imath^\sharp(\varsigma^{p+2}))=(\imath\circ\tau)^\sharp(\varsigma^{p+2})$. Without loss of generality,
we may assume $p>1$ as there is nothing to prove for $p=1$. We shall omit writing $\imath$ identifying $M$ as a submanifold of $\RR^{p+2}$, and identify $D^2$ as the unit disk in $\CC$. 

Let $\mathcal{N}$ be a closed
tubular neighborhood of $M$ in $\RR^{p+2}$, identified with
$M\times D^2$ such that $M$ is identified with $M\times\{ 0\}$ and $M\times\{1\}$
is the push-off of $M$ along $W$. By the
uniqueness of normal bundle for codimension $2$ locally flat
embedding (see \cite{KS} for the ambient dimension $\geq 5$, and
\cite[Section 9.3]{FQ} for that $=4$), we may assume
$\tilde{\tau}$ preserves $\mathcal{N}$, namely restricted to this
neighborhood, $\tilde{\tau}$ is a bundle map $M\times D^2\to M\times D^2$,
$\tilde{\tau}(x,v)\to (\tau(x), R(x).v)$, where $R(x)\in\SO(2)$.

Because $\tilde{\tau}(\Sigma)$ is still a (topological) Seifert
hypersurface, by the same argument of Lemma \ref{independent}, $R:M\to\SO(2)$ is homotopic
to the constant identity map. This
implies that $\tilde{\tau}|_N$ may be assumed to be
$\tau\times\id_{D^2}$ under the identification $\mathcal{N}\cong M\times
D^2$. Let
$X=S^{p+2}\setminus\mathring{\mathcal{N}}$, where $S^{p+2}=\RR^{p+2}\cup\{
\infty\}$. Extend $\tilde{\tau}$ to a homeomorphism of $S^{p+2}$, still
denoted as $\tilde{\tau}$, by defining
$\tilde{\tau}(\infty)=\infty$. We glue two (opposite) copies $X$, $-X$ along the boundary via
$\tilde{\tau}|_{\partial X}:\partial X\to\partial X$, and the
resulting smooth manifold is called $Y_\tau=X\cup_{\tau} (-X)$. On
the other hand, we may glue via $\id|_{\partial X}$ to obtain the
double of $X$, called $Y_\id=X\cup_\id (-X)$. Thus $Y_\tau$ is
homeomorphic to $Y_\id$ via $\tilde{\tau}\cup\id$.

Observe that $TY_\id$ is stably trivial. In fact, $X\subset
S^{p+2}=\partial D^{p+3}$, and we may push the interior of $X$
into the interior of $D^{p+3}$ so that $(X,\partial X)\subset
(D^{p+3},S^{p+2})$ is a proper embedding of pairs. We may further
assume that on the collar neighborhood of $\partial D^{p+3}$,
diffeomorphically $S^{p+2}\times I$, $X$ is identified as
$\partial X\times I$. Then doubling $D^{p+3}$ along boundary gives
a codimension $1$ smooth embedding $Y_\id\subset S^{p+3}$. Hence
clearly $TY_\id\oplus\epsilon^1$ is trivial, so $w_2(Y_\id)=0\in
H^2(Y_\id;\ZZ_2)$. Because the Stiefel-Whitney class depends only
on the homotopy type of the smooth manifold (cf. \cite{Wu}, also
\cite{MS}), it suffices to show that if $\tau$ does not preserve
$\imath^\sharp({\varsigma^{p+2}})$, then $w_2(Y_\tau)\in
H^2(Y_\tau;\ZZ_2)$ does not vanish.

Suppose on the contrary that $\tau$ does not preserve
$\imath^\sharp(\varsigma^{p+2})$, then there is some smoothly embedded loop $\alpha\subset M$
such that $\imath^\sharp(\varsigma^{p+2})|_\alpha\not\simeq \tau^*(\imath^\sharp(\varsigma^{p+2}))|_\alpha$.
Since we assumed that $\tilde{\tau}|_{\mathcal{N}}=\tau\times \id_{D^2}$
under the identification $\mathcal{N}\cong M\times D^2$, by the construction
of $\imath^\sharp(\varsigma^{p+2})$, we see
that the difference  $\tilde{\tau}_*(\varsigma^{p+2}|_{\mathcal{N}})-\varsigma^{p+2}|_{\mathcal{N}}$
between the spin structures on
$T\RR^{p+2}|_\mathcal{N}$ as an element $\in H^1(\mathcal{N};\ZZ_2)$ (cf. Section \ref{Sec-spinPrelim})
equals $\tau_*(\imath^\sharp(\varsigma^{p+2}))-(\imath^\sharp(\varsigma^{p+2}))\in H^1(M;\ZZ_2)$,
under the natural inclusion isomorphism $H^1(\mathcal{N};\ZZ_2)\to H^1(M;\ZZ_2)$.
Hence $\tilde{\tau}_*(\varsigma^{p+2}|_{\alpha\times\{1\}})\not\simeq\varsigma^{p+2}|_{\tau(\alpha)\times\{1\}}$.

As $\alpha\times\{1\}$ is null-homological in $X$ by the construction (cf. the proof of
Lemma \ref{independent}), it bounds a smoothly immersed oriented surface $j:F\looparrowright X$
such that $j(\mathring{F})\subset\mathring{X}$, and $j$ is a smooth embedding in a collar neighborhood
of $\partial F$. This can be seen by writing $\alpha$ as a product of commutators,
so there is a continuous map $F\to X$, which can be perturbed to be an immersion by the Whitney's trick.
Thus, there is a smoothly immersed closed oriented surface $\widehat{j}:K=F\cup(-F)\looparrowright Y_\tau$ defined by $j\cup (-\tilde\tau\circ j)$. 

However,
$\widehat{j}^*(TY_\tau)|_F=j^*(T\RR^{p+2})|_F$ has a spin structure
$\varsigma^{p+2}|_F$, and 
$\widehat{j}^*(TY_\tau)|_{-F}=(-\tilde{\tau}\circ j)^*(T\RR^{p+2})|_{-F}$ has a
spin structure $-\varsigma^{p+2}|_{(-\tilde{\tau}(F))}$. They
disagree along $\alpha\times \{ 1\}\subset \partial X$
(corresponding to $\tau(\alpha)\times\{ 1\}\subset\partial (-X)$).
This implies that $w_2(\widehat{j}^*(TY_\tau))\neq 0\in H^2(K;\ZZ_2)$, and
hence $w_2(Y_\tau)\neq 0\in H^2(Y_\tau;\ZZ_2)$ as it pulls back
giving a nontrivial element of $H^2(K;\ZZ_2)$. This contradicts that $Y_\id$ is homeomorphic
to $Y_\tau$, so $\tau$ has to preserve $\imath^\sharp(\varsigma^{p+2})$.
\end{proof}

\begin{remark}\label{KT-GM} We are aware that the induced spin structure
$\imath^\sharp(\varsigma^{p+2})$ can also be derived from a general 
construction of characteristic pairs (\cite{KT}, cf. also \cite{GM}, \cite{Er}).
Recall that a pair of oriented compact smooth manifolds $(W,M)$ is called \emph{characteristic}
if $M\subset W$ is a proper codimension-$2$ submanifold dual to $w_2(M)$.
The space ${\rm Char}(W,M)$ of \emph{characterizations} of $(W,M)$ consists of 
spin structures on $W\setminus M$ which does not extend across any component of $M$, admitting
a natural free transitive action $H^1(W;\ZZ_2)$. There is a function
$h:{\rm Char}(W,M)\to\mathcal{S}(M)$ equivariant under the natural
actions of $H^1(W;\ZZ_2)$ on ${\rm Char}(W,M)$ and $H^1(M;\ZZ_2)$
on $\mathcal{S}(M)$ via the homomorphism $H^1(W;\ZZ_2)\to H^1(M;\ZZ_2)$, where $\mathcal{S}(M)$ is
the space of spin structures on $M$, 
(\cite[Definition 6.1, Theorem 2.4, Lemma 6.2]{KT}).
When $W=S^{p+2}$ and $M$ is connected, ${\rm Char}(W,M)$ is a single element group whose image
under $h$ coincides with $\imath^\sharp(\varsigma^{p+2})$. This gives an alternative
 proof of Proposition \ref{spin-obstruction-top} 
 if one assumes $\tau$ extends over $S^{p+2}$ diffeomorphically rather than just homeomorphically.
When $p=2$, one can also phrase $\imath^\sharp(\varsigma^{p+2})$ in terms of the Rokhlin quadratic form, see 
Lemma \ref{Rokhlin}.
\end{remark}

Before going to the applications, we mention the following lemma which justifies the well-definedness of $\esg_\cat(\imath)$.

\begin{lemma}\label{extendTogether} Let $\imath:M\hookrightarrow\RR^{p+2}$ be a smooth embedding of an orientable closed $p$-dimensional manifold. 
Let $\tau,\tau':M\to M$ be two $\cat$-isotopic orientation-preserving $\cat$-homeomorphisms, then $\tau$ is $\cat$-extendable if and only if $\tau'$ is $\cat$-extendable over $\imath$.
\end{lemma}

\begin{proof} First assume $\tau'$ is the identity. Take a tubular neighborhood $\mathcal{N}$ of $\imath(M)$ in $\RR^{p+2}$, we have seen that $\mathcal{N}$ is diffeomorphic to $M\times D^2$. As $\tau$ is $\cat$-isotopic to
the identity, say $f_t:M\to M$ where $t\in [0,1]$, we define $\tilde\tau|:M\times D^2\to M\times D^2$ by 
$\tilde\tau(x,re^{i\theta})=(f_r(x),re^{i\theta})$, where $D^2$ is identified as the unit disk of $\CC$. Then $\tilde\tau$ is the identity restricted to $\partial\mathcal{N}\cong M\times \partial D^2$. We may further extend $\tilde\tau$ outside $\mathcal{N}$ over $\RR^{p+2}$ by the identity. This implies $\tau$ is $\cat$-extendable.

In the general case, let $\tau,\tau'$ be two $\cat$-isotopic orientation-preserving $\cat$-homeomorphisms, then $\tau^{-1}\circ\tau'$ is $\cat$-isotopic to the identity and hence $\cat$-extendable. Let $\phi:\RR^{p+2}\to\RR^{p+2}$ be
an orientation-preserving $\cat$-homeomorphic extension of $\tau^{-1}\circ\tau'$. If $\tau$ is $\cat$-extendable, say as $\tilde\tau:\RR^{p+2}\to\RR^{p+2}$, then $\tau'$ may be extended as $\tilde\tau\circ\phi$, and vice versa.
\end{proof}

\section{Embedded surfaces in $\RR^4$}\label{Sec-F_g}

In this section, we prove Theorem \ref{main-extend-F_g}. Note $\mcg_\cat(F_g)$ for $\cat=\diff,\pl,\topo$ are all canonically isomorphic to ${\rm Out}(\pi_1(F_g))$ due to the Dehn-Nielsen-Baer theorem (cf. \cite{Iv}).

Let $\mathcal{S}(F_g)$ be the space of spin structures on a closed
connected oriented surface $F_g$ of genus $g$. There is a surjective map:
$$\mathcal{S}(F_g)\stackrel{[.]}\longrightarrow \Omega^\Spin_2\stackrel\Arf\longrightarrow\ZZ_2,$$
where $\Omega^\Spin_2$ is the second spin cobordism group and $\Arf$ is the Arf isomorphism. More precisely,
for any $\sigma\in\mathcal{S}(F_g)$, there is an associated nonsingular quadratic function $q_\sigma:H_1(F_g;\ZZ_2)\to\ZZ_2$,
such that $q_\sigma(\alpha)=0$ (resp. $1$) if the spin
structure on $F_g$ restricted to the bounding (resp. Lie-group) spin structure on $\alpha$. 
Note $q_\sigma(\alpha+\beta)=q_\sigma(\alpha)+q_\sigma(\beta)+\alpha\cdot\beta$ where $\alpha\cdot\beta$ is the $\ZZ_2$-intersection number, and $\sigma=\sigma'$ if and only if $q_{\sigma}=q_{\sigma'}$. Thus $\Arf([\sigma])$ is defined as the Arf invariant of the nonsingular quadratic form $q_\sigma$. Recall that for a nonsingular quadratic form $q$ on $V\cong\ZZ_2^{\oplus 2g}$, $\Arf(q)$ is $0$ (resp. $1$) if and only if $q$ vanishes on exactly $2^{2g-1}+2^{g-1}$ (resp. $2^{2g-1}-2^{g-1}$) elements, (cf. \cite[Appendix]{Ki}). Correspondingly, $\mathcal{S}(F_g)$ is a disjoint union:
$$\mathcal{S}(F_g)=\mathcal{B}_g\sqcup\mathcal{U}_g,$$
of bounding and unbounding spin structures. We denote the cardinal numbers of $\mathcal{B}_g$, $\mathcal{U}_g$ as $b_g$, $u_g$, respectively.

\begin{lemma}\label{cardBU} For $g\geq 1$, $b_g=2^{2g-1}+2^{g-1}$ and
$u_g=2^{2g-1}-2^{g-1}$.\end{lemma}

\begin{proof} For $g=1$, it is well-known that the only unbounding spin structure on $F_1=T^2$ is the Lie-group spin structure, so $b_1=3$, $u_1=1$. In general, any pair of two spin structures $\sigma_g\in\mathcal{S}(F_g)$, $\delta\in\mathcal{S}(T^2)$ determines a bounding (resp. unbounding) spin structure on $F_{g+1}\cong F_g\,\#\, T^2$ if and only if $\Arf([\sigma_g])$ and $\Arf([\delta])$ have the same (resp. distinct) parity. This implies $b_{g+1}=b_1\times b_g+ u_1\times u_g=3\,b_g+u_g$, and $u_{g+1}=b_1\times u_g + u_1\times b_g= 3\,u_g+b_g$, so $b_{g+1}-u_{g+1}=2\,(b_g-u_g)=\cdots=2^g\,(b_1-u_1)=2^{g+1}$. Using $b_g-u_g=2^g$ and $b_g+u_g=2^{2g}$, we see 
$b_g=2^{2g-1}+2^{g-1}$, $u_g=2^{2g-1}-2^{g-1}$.
\end{proof}

There is a natural action of $\mcg_\cat(F_g)$ on $\mathcal{S}(F_g)$, 
where any $[\tau]\in\mcg_\cat(F_g)$ acts as the pull-back $\tau^*:\mathcal{S}(F_g)\to \mathcal{S}(F_g)$.

\begin{lemma}\label{actionBU}  For $g\geq 1$, $\mcg_\cat(F_g)$ acts invariantly and transitively on
$\mathcal{B}_g$ and $\mathcal{U}_g$.
\end{lemma}

\begin{proof} The invariance of the $\mcg_\cat(F_g)$-action on $\mathcal{B}_g$ and $\mathcal{U}_g$ follows immediately from, for example, counting vanishing elements of the associated quadratic forms $q_\sigma, q_{\tau^*(\sigma)}$ for $\sigma\in \mathcal{S}(F_g)$ and $[\tau]\in\mcg_\cat(F_g)$. It suffices to prove the transitivity of the action.
We argue by induction on $g\geq 1$.

When $g=1$, $F_1$ is $T^2\cong S^1_1\times S^1_2$, and $\mcg_\cat(T^2)\cong\SL(2,\ZZ)$ is generated be the Dehn twists $D_1, D_2$ along the first and second factors. It is straightforward to check that $\mcg_\cat(T^2)$ acts transitively on $\mathcal{B}_1$ and $\mathcal{U}_1$.

Suppose for some $g\geq 1$, $\mcg_\cat(F_g)$ acts transitively on $\mathcal{B}_g$ and $\mathcal{U}_g$ for some $g\geq 1$. To see $\mcg_\cat(F_{g+1})$ acts transitively on $\mathcal{B}_{g+1}$, let $\sigma,\sigma'\in\mathcal{B}_{g+1}$. Pick a connect sum decomposition $F_{g+1}\cong F_g\# T^2$, which induces a decomposition $H_1(F_{g+1};\ZZ_2)\cong H_1(F_g;\ZZ_2)\oplus H_1(T^2;\ZZ_2)$. Then $\sigma$ determines spin structures $\sigma_g\in\mathcal{S}(F_g)$ and $\delta\in \mathcal{S}(T^2)$ so that $[\sigma]=[\sigma_g]+[\delta]$ in
$\Omega^\Spin_2$, and similarly $\sigma'$ determines $\sigma'_g$, $\delta'$ so that $[\sigma']=[\sigma'_g]+[\delta']$. If $[\sigma_g]=[\sigma'_g]$, and hence $[\delta]=[\delta']$, then by the induction assumption there are $[\tau_g]\in\mcg_\cat(F_g)$ and $[\phi]\in\mcg_\cat(T^2)$ such that $\tau_g^*(\sigma_g)=\sigma'_g$, $\phi^*(\delta)=\delta'$. Then one finds an element
$[\tau]\in\mcg_\cat(F_{g+1})$, where $\tau=\tau_g\#\phi$, such that $\tau^*([\sigma])=\sigma'$.

Now we consider the case if $[\sigma_g]\neq[\sigma'_g]$, and hence $[\delta]\neq[\delta']$. Thus one of $\delta,\delta'\in\mathcal{S}(T^2)$ is the Lie-group spin structure, and the other is a spin-boundary, but there always exists some nontrivial $[\alpha]\in H_1(T^2;\ZZ_2)$ such that $\delta|_\alpha=\delta'|_\alpha$. For any $[\beta]\in H_1(T^2;\ZZ_2)$ with $\alpha\cdot\beta=1$, that $[\delta]\neq[\delta']$ implies $\delta|_\beta\neq\delta'|_\beta$. On the other hand, there is some nontrivial $[\gamma]\in H_1(F_g;\ZZ_2)$, such that $\sigma_g|_\gamma\neq\sigma'_g|_\gamma$. Let $[\tilde\beta]=[\beta]+[\gamma]\in H_1(F_{g+1};\ZZ_2)$, we have $\alpha\cdot\tilde\beta=1$, and the difference $\sigma-\sigma'\in H^1(F_g;\ZZ_2)$ vanishes on $[\alpha]$ and $[\tilde\beta]$. We may take two simple closed curve representatives $\alpha, \tilde\beta\subset F_{g+1}$ such that $\alpha\cap\tilde\beta$ is a single point. A regular neighborhood of $\alpha\cup\tilde\beta$ on $F_{g+1}$ is a punctured torus $\tilde{T}\setminus *$, which gives another connected sum decomposition of $F_{g+1}=\tilde{F}_g\#\tilde{T}$. It is clear that with respect to this decomposition, the induced spin structures $\tilde\sigma_g,\tilde\sigma'_g\in\mathcal{S}(F_g)$, $\tilde\delta,\tilde\delta'\in\mathcal{S}(F_g)$ satisfy $[\tilde\sigma_g]=[\tilde\sigma'_g]$, $[\tilde\delta]=[\tilde\delta']$, so we apply the previous case to obtain some $[\tilde\tau]\in\mcg_\cat(F_{g+1})$, such that $\tilde\tau^*([\sigma])=\sigma'$. This means $\mcg_\cat(F_{g+1})$ acts transitively on $\mathcal{B}_{g+1}$.

The proof for the transitivity of the $\mcg_\cat(F_{g+1})$-action on $\mathcal{U}_{g+1}$ is similar, so we complete the induction.\end{proof}

\begin{proof}[Proof of Theorem \ref{main-extend-F_g}]
Because $\mcg_\topo(F_g)$ are represented by self-diffeomorphisms, by
Proposition \ref{spin-obstruction-top}, any element in $\esg_\topo(\imath)$
preserves $\imath^\sharp(\varsigma^{4})\in\mathcal{B}_g\subset\mathcal{S}(F_g)$.
On the other hand, $\mcg_\topo(F_g)$ acts transitively on $\mathcal{B}_g$ (Lemma
\ref{actionBU}).
Therefore, $[\mcg_\topo(F_g):\esg_\topo(\imath)]\geq |\mathcal{B}_g|=2^{2g-1}+2^{g-1}$,
(Lemma \ref{cardBU}).
\end{proof}

For a smoothly embedded oriented closed surface $\imath: F_g\hookrightarrow\RR^4$, the Rokhlin quadratic
form $q_\imath:H_1(F_g;\ZZ_2)\to \ZZ_2$ is defined so that for any smoothly embedded
subsurface $P\subset\RR^4$ with $\partial P\subset \imath(F_g)$,
$\mathring{P}\subset \RR^4\setminus\imath(F_g)$ and
 transverse to $\imath(F_g)$ along $\partial P$, $q_\imath([\partial P])$ is the $\bmod\,2$ number
 of points in $P\cap P'$, where $P'$ is a smooth perturbed copy of $P$ so that $\partial P'\subset \imath(F_g)$
 is disjoint parallel to $\partial P$, and that $\mathring{P}'$ is transverse to $\mathring{P}$, (\cite{Ro}).
In dimension $4$, the induced spin structure $\imath^\sharp(\varsigma^4)$ is related to it as follows.

\begin{lemma}\label{Rokhlin} The quadratic form $q_{\imath^\sharp(\varsigma^4)}$ coincides with the Rokhlin form $q_\imath$.\end{lemma} 

\begin{proof} To see this, consider a smoothly embedded surface $P\subset\RR^4$ as in the definition of $q_\imath$. Note that the normal vector field of $\partial P$ in $P$ is parallel to the vector field $W$ as in Section \ref{Sec-inducedSpin}. There is a trivialization defined by a frame field $(U,V,W,H)|_{\partial P}$ such that for any $x\in \partial P$, $U_x\in T(\partial P)|_x$, $V_x\in N_{\imath(F_g)}(\partial P)|_x$, $W_x\in N_{P}(\partial P)|_x$, and $H_x$ is a complementary vector orthogonal to $U_x,V_x,W_x$. Now $(U,V,W,H)\simeq\varsigma^4|_{\partial P}$ if and only if it extends over $T\RR^4|_P$. Since $(U,W)|_{\partial P}$ does not (stably) extend over $TP$, $(U,V,W,H)\simeq\varsigma^4|_{\partial P}$ if and only if $(V,H)|_{\partial P}$ fails to extend (stably) over $N_{\RR^4}(P)$, i.e. $|P\cap P'|$ is odd. In this case, $\imath^\sharp(\varsigma^4)|_{\partial P}$ is by definition given by $(U,V)|_{\partial P}$ which is the Lie-group spin structure. We conclude $q_\imath([\partial P])=1$ if and only if 
$q_{\imath^\sharp(\varsigma^4)}([\partial P])=1$.\end{proof}

It is clear from its definition that the Rokhlin form is invariant under the action of $[\tau]\in\mcg_\cat(F_g)$ if $\tau$ extends diffeomorphically. This would imply $[\mcg_\diff(F_g):\esg_\diff(\imath)]\geq 2^{2g-1}+2^{g-1}$ following Lemmas \ref{Rokhlin}, \ref{cardBU}, \ref{actionBU}. On the other hand, Hirose showed in \cite{Hi} that for an unknotted smooth embedding $\imath:F_g\hookrightarrow\RR^4$, i.e. which bounds a smoothly embedded handlebody, $[\tau]\in\esg_\diff(\imath)$ if and only if $\tau$ preserves the Rokhlin quadratic form, (cf. also \cite{Mo} for $g=1$). Noting that $\esg_\diff(\imath)\leq\esg_\pl(\imath)\leq\esg_\topo(\imath)$
under the natural isomorphism $\mcg_\diff(F_g)\cong\mcg_\pl(F_g)\cong\mcg_\topo(F_g)$, we have
$[\mcg_\cat(F_g):\esg_\cat(\imath)]=2^{2g-1}+2^{g-1}$ for unknotted embeddings.

Corollary \ref{noExtension} is an easy consequence of Lemma \ref{actionBU}:

\begin{proof}[{Proof of Corollary \ref{noExtension}}] Observe that the action of $\mcg_\cat(F_g)$ on $\mathcal{S}(F_g)$ descends to an action of a group $\Gamma<\Aut^+(H_1(F_g;\ZZ_2))$: indeed, if $\tau$ projects to $\id\in\Aut^+(H_1(F_g;\ZZ_2))$, $q_{\tau^*\sigma}([\alpha])=q_\sigma(\tau_*[\alpha])=q_\sigma([\alpha])$, for any $[\alpha]\in H_1(F_g;\ZZ_2)$, so $\tau^*\sigma=\sigma$ for any $\sigma\in\mathcal{S}(F_g)$. $\Gamma$ is a finite group isomorphic to ${\rm Sp}(2g,\ZZ_2)$ as it preserves the $\ZZ_2$-intersection form. Then Lemma \ref{actionBU} implies $\Gamma$ acts transitively on $\mathcal{B}_g$, so for any $\sigma\in\mathcal{B}_g$, ${\rm Stab}_\Gamma(\sigma)<\Gamma$ has index $b_g=2^{2g-1}+2^{g-1}$.
Since $\id\in {\rm Stab}_\Gamma(\sigma)$ for all $\sigma\in\mathcal{B}_g$, the subset:
$$W=\bigcup_{\sigma\in\mathcal{B}_g}{\rm Stab}_\Gamma(\sigma)\subset \Gamma,$$
has at most $b_g(\frac{|\Gamma|}{b_g}-1)+1<|\Gamma|$ elements. Thus for any $[\tau]\in\Gamma\setminus W$, $\tau$ does not fix any $\sigma\in\mathcal{B}_g$. In particular, for any smooth embedding $\imath:F_g\hookrightarrow\RR^4$, $\imath^\sharp(\varsigma)\in\mathcal{B}_g$ will not be invariant under $\tau$. By Proposition \ref{spin-obstruction-top},
$[\tau]\not\in\esg_\topo(\imath)$.
\end{proof}

\section{Embedded $p$-tori in $\RR^{p+2}$}\label{Sec-T^p}

In this section, we prove Theorem \ref{main-extend-T^p} and its corollaries.

Let $T^p$ be the standard $p$-dimensional torus. Fix a parametrization
$T^p=S^1_1\times\cdots\times S^1_p$, where each $S^1_i$ is a copy of the unit circle $S^1\subset\CC$.
We start by some general facts about $\mcg_\cat(T^p)$ and its action on the space $\mathcal{S}(T^p)$
of spin structures on $T^p$.

For any $p\geq 2$, ${\rm Homeo}^+_\cat(T^p)$ has a \emph{modular subgroup} $\Mod(T^p)\cong \SL(p,\ZZ)$ generated by elements represented by the Dehn twists $\tau_{i,j}$ ($1\leq i,j\leq p$, $i\neq j$) along the $i$-th factor in the $S^1_i\times S^1_j$ direction, namely, $\tau_i(u_1,\cdots,u_p)=(u_1,\cdots,u_{j-1},u_iu_j,u_{j+1},\cdots,u_p)$. $\Mod(T^p)$ may identically be regarded as a subgroup of $\mcg_\cat(T^p)$ under the natural quotient $\pi_0:{\rm Homeo}^+_\cat(T^p)\to\mcg_\cat(T^p)$. Thus the action of $\mcg_\cat(T^p)$ on $H_1(T^p;\ZZ)$ induces a splitting sequence of groups:
$$1\to\mathscr{I}_\cat(T^p)\to\mcg_\cat(T^p)\to\SL(p,\ZZ)\to 1,$$
as $\Aut^+(H_1(T^p;\ZZ))\cong\SL(p,\ZZ)$, (which holds trivially for $p=1$ as well). In other words, $\mcg_\cat(T^p)=\mathscr{I}_\cat(T^p)\rtimes\Mod(T^p)$.
It is well-known that $\mcg_\cat(T^2)=\Mod(T^2)$ (cf. \cite{Iv}), and $\mcg_\cat(T^3)=\Mod(T^3)$ follows from general results of Hatcher for Haken $3$-manifolds (\cite{Ha1}, \cite{Ha3}). While the case $p=4$ remains mysterious, for $p\geq 5$, the splitting is known to be nontrivial and $\mcg_\cat(T^p)$ are different for various $\cat$'s. Specifically, a theorem of Hatcher (\cite[Theorem 4.1]{Ha2}, cf. also \cite{HS}) implies $\mathscr{I}_\cat(T^p)$ ($p\geq5$) is an infinitely generated abelian group, which can be regarded as a $\SL(p,\ZZ)$-module with the following decomposition:
$$\mathscr{I}_\diff(T^p)\,\cong\,\mathscr{W}_p\,\oplus\,H^2(T^p;\ZZ_2)\,\oplus\,\bigoplus_{i=1}^p\,H^i(T^p;\Gamma_{i+1}),$$
and $\mathscr{I}_\pl(T^p)\cong\mathscr{W}_p\,\oplus\,H^2(T^p;\ZZ_2)$, $\mathscr{I}_\topo(T^p)\cong\mathscr{W}_p$ as induced by the forgetting quoients. Here $\mathscr{W}_p\cong\ZZ_2[t_1,t_1^{-1},\cdots,t_p,t_p^{-1}]\,/\,\ZZ_2[t_1+t_1^{-1},\cdots,t_p+t_p^{-1}]\cong\ZZ_2^{\oplus\infty}$ has the natural action induced by that of $\SL(p,\ZZ)$ on the monomials, and $\Gamma_i$ is the $i$-th Kervaire-Milnor group of homotopy spheres which is finite abelian, $i\geq0$, and the $\SL(p,\ZZ)$ acts on $H^2(T^p;\ZZ_2)$, $H^i(T^p;\Gamma_{i+1})$ naturally as usual.

As the space of spin structures $\mathcal{S}(T^p)$ is an affine $H^1(T^p;\ZZ_2)$, there is a Lie-group spin structure and $2^p-1$ non-Lie-group spin structures. Denote the subset of non-Lie-group spin structures as $\mathcal{S}^\star(T^p)$.  

The lower bound in Theorem \ref{main-extend-T^p} follows from the lemma below.

\begin{lemma}\label{actionT^p} $\mcg_\cat(T^p)$ fixes the Lie-group
spin structure of $T^p$, and acts transitively on $\mathcal{S}^\star(T^p)$. Hence $[\mcg_\cat(T^p):\esg_\cat(\imath)]\geq 2^p-1$ if $\imath:T^p\hookrightarrow\RR^{p+2}$ induces 
a non-Lie-group spin structure $\imath^\sharp(\varsigma^{p+2})$ on $T^p$.
\end{lemma}

\begin{proof} For the standard parametrization $u=(u_1,\cdots,u_p)$ of $T^p=S^1_1\times\cdots\times S^1_p$, the Lie group spin structure $\sigma_0\in\mathcal{S}(T^p)$ is represented by the standard framing $(\frac{\partial}{\partial u_1},\cdots,\frac{\partial}{\partial u_p})$ over $T^p$, so for any $\tau\in\Mod(T^p)$, $\tau^{-1}_*(\frac{\partial}{\partial u_1},\cdots,\frac{\partial}{\partial u_p})|_u=(\frac{\partial}{\partial u_1},\cdots,\frac{\partial}{\partial u_p})|_{\tau(u)}\cdot A$ for the matrix $A\in\SL(p,\ZZ)$ defining $\tau$ for any $u\in T^p$. This means pulling-back by $\tau$ fixes the framing over $T^p$ up to homotopy, so $\tau^*(\sigma_0)=\sigma_0$. On the other hand, $\mathscr{I}_\cat(T^p)$ fixes $\sigma$ since the action of $\mcg_\cat(T^p)$ descends to $\Aut^+(H_1(T^p))\cong\SL(p,\ZZ)$. Thus $\mcg_\cat(T^p)$ fixes $\sigma_0$.

Let $\sigma',\sigma''\in \mathcal{S}^\star(T^p)$, the differences $\sigma'-\sigma_0,\sigma''-\sigma_0\in H^1(T^p;\ZZ_2)\setminus \{0\}$. As $\mcg_\cat(T^p)$ acts transitively on $H^1(T^p;\ZZ_2)\setminus \{0\}$ and fixes $\sigma_0$, there is some $[\tau]\in\mcg_\cat(T^p)$ such that $\tau^*(\sigma')=\sigma''$. Thus $\mcg_\cat(T^p)$ acts transitively on $\mathcal{S}^\star(T^p)$.

Finally, by Proposition \ref{spin-obstruction-top}, $\tau\in\esg_\cat(\imath)$ only if $\tau$ fixes $\imath^\sharp(\varsigma^{p+2})$, so the transitivity implies $[\mcg_\cat(T^p):\esg_\cat(\imath)]\geq |\mathcal{S}^\star(T^p)|=2^p-1$ if $\imath^\sharp(\varsigma^{p+2})\in\mathcal{S}^\star(T^p)$.
\end{proof}

A little more can be said about $\esg_\cat(\imath)$ for general smooth embeddings of $T^p$ into $\RR^{p+2}$.

\begin{lemma}\label{IcapE} For $p\geq 1$, and for any smooth embedding $\imath:T^p\hookrightarrow\RR^{p+2}$, $\mathscr{I}_\cat(T^p)\cap\esg_\cat(\imath)$ has finite index in $\mathscr{I}_\cat(T^p)$. Moreover,
$\mathscr{I}_\topo(T^p)\leq\esg_\topo(\imath)$.
\end{lemma}

\begin{proof} Without loss of generality, we may assume $p\geq 4$ as $\mathscr{I}_\cat(T^p)$ is trivial when $p\leq 3$.

First suppose $p\geq 5$. In this case, it suffices to show $\mathscr{W}_p\leq\esg_\diff(\imath)$. Let $[\tau]\in\mathscr{W}_p$ where $\tau$ is a diffeomorphic representative. By Remark (4) of \cite[Theorem 4.1]{Ha2},
$\tau$ is smoothly concordant to $\id$, namely, there is a diffeomorphism $f:T^p\times [0,1]\to T^p\times [0,1]$,
such that $f|_{T^p\times\{0\}}=\tau$, $f|_{T^p\times\{1\}}=\id_{T^p}$. Let $f_T$, $f_I$ be the first
and the second component of $f$, respectively, i.e. such that $f(u,r)=(f_T(u,r),f_I(u,r))$.
Pick be a tubular neighborhood $\mathcal{N}\cong T^p\times D^2$ of $\imath(T^p)$ in $\RR^{p+2}$. Identify $D^2$ as the unit disk of $\CC$, and define $\tilde\tau|:T^p\times D^2\to T^p\times D^2$ by $\tilde\tau(u,\,r\,e^{i\theta})=(f_T(u,r),
\,f_I(u,r)\,e^{i\theta})$. It is clear that $\tilde\tau|$ is an orientation-preserving diffeomorphism which restrict to $T^p\times \partial D^2$ as identity. We may define an orientation-preserving diffeomorphism $\tilde\tau:\RR^4\to \RR^4$ by extending $\tilde\tau|$ as identity outside $\mathcal{N}$, which extends $\tau$. This shows $[\tau]\in\esg_\diff(\imath)$.

For $p=4$, let $[\tau]\in\mathscr{I}_\cat(T^4)$ where $\tau$ is a $\cat$-homeomorphic representative. Pick be a tubular neighborhood $\mathcal{N}\cong T^4\times D^2$ of $\imath(T^4)$ in $\RR^6$. We first define $\tilde\tau:T^4\times D^2(\frac12)\to T^4\times D^2(\frac12)$ as $\tau\times\id_{D^2(\frac12)}$, where $D^2(\frac12)$ is the disk of radius one half. $\tilde\tau|$ restricted to $T^4\times \partial D^2(\frac12)$ may be regarded as an element of $\mathscr{I}_\cat(T^5)$. If it lies in $\mathscr{W}_5$, then there is a $\cat$-concordance $f:T^5\times [0,1]\to T^5\times [0,1]$ between $\tilde\tau|_{T^4\times \partial D^2(\frac12)}$ and the identity obtain by joining a $\cat$-isotopy between $\tilde\tau|_{T^4\times \partial D^2(\frac12)}$ and a diffeomorphic representative $\phi\in[\tau]$ with a smooth concordance between $\phi$ and the identity. As $T^5\times [0,1]\cong T^4\times (D^2\setminus \mathring{D}^2(\frac12))$, we may extend $\tilde\tau|$ using the $\cat$-concordance $f$ over $\mathcal{N}$ such that $\tilde\tau|_{\partial\mathcal{N}}$ is the identity. Further extend $\tilde\tau|$ outside $\mathcal{N}$ by the identity, we see $[\tau]\in\esg_\cat(\imath)$. This means the preimage of $\mathscr{W}_5$ under:
$$\mathscr{I}_\cat(T^4)\to\mathscr{I}_\cat(T^5),$$
defined by $[\tau]\to[\tau\times\id_{S^1}]$, is contained in $\esg_\cat(\imath)$. Since $\mathscr{W}_5$ has finite index in $\mathscr{I}_\cat(T^5)$, we conclude
$\mathscr{I}_\cat(T^4)\cap\esg_\cat(\imath)$ has finite index in $\mathscr{I}_\cat(T^4)$ as well. Moreover,
$\mathscr{I}_\topo(T^4)\leq\esg_\topo(\imath)$ since $\mathscr{W}_5=\mathscr{I}_\topo(T^5)$.
\end{proof}

We proceed to consider unknotted embeddings of $T^p$ into $\RR^{p+2}$. These have been defined and
studied in \cite{DLWY}. We recall the notion and properties enough
for our use here. Regard $S^1$ and $D^2$ as the unit circle and
the unit disk of $\CC$, respectively. The standard basis of $\RR^n$ is
$(\vec{\varepsilon}_1,\cdots,\vec{\varepsilon}_n)$, and the
$m$-subspace spanned by
$(\vec{\varepsilon}_{i_1},\cdots,\vec{\varepsilon}_{i_m})$ will be
written as $\RR^m_{i_1,\cdots,i_m}$, and hence
$\RR^n=\RR^n_{1,\cdots,n}\subset \RR^{n+1}$.

\begin{example}[The standard model]\label{stdmodel}
Let $\imath_0:{\rm pt}=T^0\to\RR^2$ be $\imath_0({\rm pt})=0$ by
convention. Inductively suppose $\imath_{p-1}$ has been constructed
for some $p\geq1$ such that $\imath_{p-1}(T^{p-1})\subset \mathring{D}^p\subset\RR^p_{2,\cdots,p+1}$. Denote
the rotation of $\RR^{p+2}$ on the subspace $\RR^2_{2,p+2}$
 of angle $\arg(u)$ as $\rho_p(u)\in\SO(p+2)$, for any $u\in S^1$, and we may
 define
 $\imath_p:T^p=T^{p-1}\times S^1_p$
 as: $$\imath_p(v,u)=\rho_p(u)(\frac12\cdot\vec{\varepsilon}_2+\frac{1}{4}\cdot\imath_{p-1}(v)).$$
This explicitly describes an embedding of
$T^p=S^1_1\times\cdots\times S^1_p$ into
$\RR^{p+1}_{2,\cdots,p+2}$. In Figure \ref{figStdModel}, the images of 
$\imath_{p-1}$ and $\imath_p$
are schematically presented on the left and the right respectively. One may imagine
$\vec{\varepsilon}_1$ points perpendicularly outward the page. Observe that the image of $T^p$ is
invariant under $\rho_p(u)$.
\end{example}

\begin{figure}[htb]
\centering
\psfrag{a}[]{\scriptsize{$\mathbb{R}^{p-1}_{3,\cdots,{p+1}}$}}
\psfrag{b}[]{\scriptsize{$\vec{\varepsilon}_2$}}
\psfrag{c}[]{\scriptsize{$\vec{\varepsilon}_{p+2}$}}
\psfrag{d}[]{\scriptsize{$\imath_{p-1}(T^{p-1})$}}
\psfrag{e}[]{\scriptsize{$\frac12\cdot\vec{\varepsilon}_2+\frac{1}{4}\cdot$\scriptsize{$\imath_{p-1}(T^{p-1})$}}}
\psfrag{f}[]{\scriptsize{$\imath_p(T^p)$}} 
\psfrag{g}[]{\scriptsize{$S^1_p$}}
\includegraphics[scale=.7]{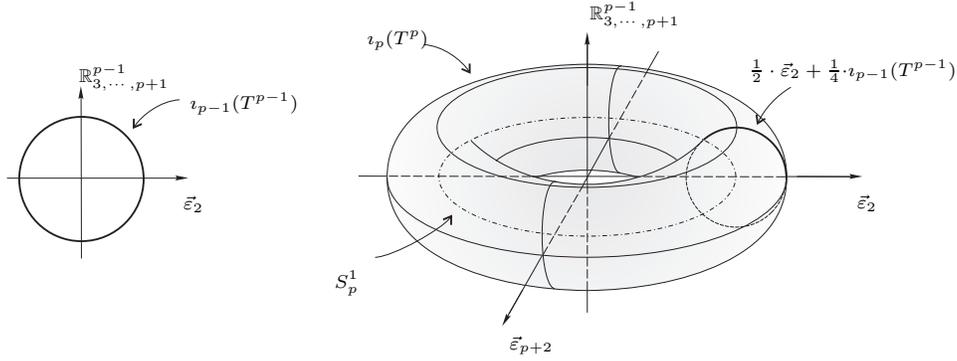}
\caption{The standard model.}\label{figStdModel}
\end{figure}

\begin{definition}\label{ukT} An embedding $\imath:T^p\hookrightarrow \RR^{p+2}$ is
called \emph{unknotted} if there is a diffeomorphism
$g:\RR^{p+2}\to\RR^{p+2}$ such that $\imath$ and $g\circ\imath_p$
have the same image, i.e. $\imath(T^p)=g\circ\imath_p(T^p)$.
\end{definition}

\begin{lemma}\label{notLie} For any unknotted embedding $\imath:T^p\hookrightarrow
\RR^{p+2}$, the induced spin structure $\imath^\sharp(\varsigma^{p+2})$ is not the Lie-group
spin structure on $T^p$.\end{lemma}

\begin{proof} One can easily see that the standard embedding $\imath_p:T^p=
S^1_1\times\cdots\times S^1_p\hookrightarrow\RR^{p+2}$ can
be extended to an embedding from $D^2\times T^{p-1}=
D^1_1\times S^1_2\times\cdots\times S^1_p$ to $\RR^{p+2}$, for $p\geq 1$,
using an induction argument. Thus $\imath$ also has a Seifert hypersurface
$\Sigma\subset\RR^{p+2}$ diffeomorphic to $D^2\times T^{p-1}$. From the proof
of Proposition \ref{spin-obstruction-top}, $(T^p,\imath^\sharp(\varsigma^{p+2}))$ is
the spin boundary of a spin structure on $\Sigma$. However, the spin structures on
$\Sigma\cong D^2\times T^{p-1}$ are $\varsigma^2\oplus\sigma$, where $\sigma\in\mathcal{S}(T^{p-1})$,
and these induce $\partial\varsigma^2\oplus\sigma$ on $\partial\Sigma
\cong S^1\times T^{p-1}$, which disagree with the
Lie-group spin structure along the loop $S^1\times *$.
\end{proof}

\begin{proof}[{Proof of Theorem \ref{main-extend-T^p}}]
Lemma \ref{actionT^p} proves $[\mcg_\topo(T^p):\esg_\topo(\imath)]\geq 2^p-1$. To see that 
any unknotted embedding $\imath:T^p\hookrightarrow\RR^{p+2}$ realizes the lower bound, note $\mcg_\topo(T^p)=\mathscr{I}_\topo(T^p)\rtimes
\Mod(T^p)$. By Lemma \ref{IcapE}, $\mathscr{I}_\topo(T^p)\leq \esg_\topo(\imath)$. 
On the other hand, \cite[Theorem 1.4]{DLWY} showed $\Mod(T^p)$ (denoted as $\Aut(T^p)$ there)
has a subgroup of index $2^p-1$ which is diffeomorphically extendable. Therefore,
$[\mcg_\topo(T^p):\esg_\topo(\imath)]\leq 2^p-1$, and hence the index is exactly $2^p-1$.
\end{proof}

\begin{proof}[{Proof of Corollary \ref{extend-T^3}}] This follows from Rokhlin's theorem that any closed spin $4$-manifold $X$ has signature $0$ modulo $16$, (cf. \cite[Theorem III.1.1]{Ki}). In fact, if $\Sigma$ is a Seifert hypersurface of $\imath$, the proof of Proposition
\ref{spin-obstruction-top} implies $\Sigma$ has a spin structure inducing $\imath^\sharp(\varsigma^5)$ on 
the boundary $T^3$. If it is the Lie-group spin structure, one can find a compact spin $4$-manifold $N$ of signature
$8\bmod16$ (cf. \cite[Chapter V]{Ki}, also \cite[Proposition 6.1]{SST}) with
$\partial N=T^3$, such that one can glue $\Sigma$ and $N$ along boundary
to obtain a closed spin $4$-manifold $X$. Then ${\rm sig}(\Sigma)+8\equiv{\rm sig}(\Sigma)+{\rm sig}(N)={\rm sig}(X)\equiv0\bmod16$ would imply $\Sigma$ has signature $8\bmod16$, which violates the assumption. Thus from $\imath^\sharp(\varsigma^5)$ is not the Lie-group spin structure on $T^3$, and Theorem \ref{main-extend-T^p} holds in this case. Note also that in this case, any Seifert hypersurface of $\imath$ has signature $0\bmod16$, (cf. for example,
\cite[Proposition 6.1]{SST}).
\end{proof}

To prove Corollary \ref{ukT-DiffPL}, we need an elementary lemma in group theory.

\begin{lemma}\label{group} If $G$ is a subgroup of a semi-direct product of groups $N\rtimes H$, then $[N\rtimes H:G]\,\leq\,[N:N\cap G]\cdot[H:H\cap G]$.\end{lemma}
\begin{proof} Let $N'=N\cap G$, $H'=H\cap G$. Clearly $H'$ preserves $N'$ under the conjugation, so the subgroup $N'H'$ is also a semi-direct product. Note $[NH:N'H']=[NH:NH']\cdot[NH':N'H']$. As $N$ is normal in both $NH$ and $NH'$, quotienting out $N$ yields $[NH:NH']=[H:H']$. Because $N\cap N'H'=N'$ as $N'H'$ is a semi-direct product, the map
$N\to NH'/N'H'$ descends to a bijection $N/N'\to NH'/N'H'$ between the cosets, so $[NH':N'H']=[N:N']$. Thus $[NH:G]\leq[NH:N'H']=[N:N']\cdot[H:H']$.
\end{proof}

\begin{proof}[{Proof of Corollary \ref{ukT-DiffPL}}] 
$[\mcg_\cat(T^p):\esg_\cat(\imath)]\geq2^p-1$ follows from Lemmas \ref{actionT^p}, \ref{notLie}. By Lemma \ref{IcapE}, $[\mathscr{I}_\cat(T^p):\mathscr{I}_\cat(T^p)\cap\esg_\cat(\imath)]$ is finite. By \cite[Theorem 1.4]{DLWY}, $[\Mod(T^p):\Mod(T^p)\cap\esg_\cat(\imath)]$ is finite. Therefore, as $\mcg_\cat(T^p)=\mathscr{I}_\cat(T^p)\rtimes \Mod(T^p)$, $[\mcg_\cat(T^p):\esg_\cat(\imath)]$ is also finite by Lemma \ref{group}.
Note clearly $[\mcg_\cat(T^p):\esg_\cat(\imath)]=2^p-1$ when $p\leq3$.
\end{proof}

\bibliographystyle{amsalpha}

\bigskip
\textsc{Peking University,
Beijing, 100871, P.~R. China.}

\textit{E-mail address:} \texttt{dingfan@math.pku.edu.cn}

\bigskip 
\textsc{University of California,
Berkeley, CA 94720, USA.}

\textit{E-mail address:} \texttt{yliu@math.berkeley.edu}

\bigskip
\textsc{Peking University,
Beijing, 100871, P.~R. China.}

\textit{E-mail address:} \texttt{wangsc@math.pku.edu.cn}

\bigskip 
\textsc{University of California,
Berkeley, CA 94720, USA.}

\textit{E-mail address:} \texttt{jgyao@math.berkeley.edu}

\end{document}